\documentclass[11pt]{article}

\usepackage[margin=1in]{geometry}

\usepackage{lmodern}
\usepackage[T1]{fontenc}

\usepackage[english]{babel} 
\usepackage[utf8]{inputenc}

\usepackage{mathrsfs} 
\usepackage{amsfonts}
\usepackage{marvosym} 
\usepackage{amsmath,amssymb,amsthm}

\usepackage{bm}

\usepackage[usenames, dvipsnames, svgnames]{xcolor}
\definecolor{vert}{RGB}{15,120,5}
\definecolor{gris}{RGB}{128,128,128}
\definecolor{bleu}{RGB}{0,50,150}
\definecolor{rouge}{RGB}{149,24,24}
\usepackage[colorlinks = true,linkcolor=bleu,citecolor=vert]{hyperref}

\usepackage[nameinlink]{cleveref} 
\crefname{equation}{}{}
\usepackage{bbm}

\usepackage{epsfig}

\usepackage{tikz-cd} 

\usepackage{stmaryrd}
\usepackage{manfnt}
\usepackage{multicol}
\usepackage{array}
\usepackage{multirow}

\usepackage{fancyhdr} 
\usepackage{fancybox} 

\usepackage{calligra}
\title{Étale motives of geometric origin}
\author{Raphaël Ruimy, Swann Tubach}
\date{}

\usepackage{tikz}
\usetikzlibrary{calc, intersections}

\usepackage{adjustbox}

\theoremstyle{plain}
\newtheorem{thm}{Theorem}[subsection]
\newtheorem{prop}[thm]{Proposition}

\newtheorem*{thm*}{Theorem}
\newtheorem*{mainthm}{Main Theorem}
\newtheorem*{prop*}{Proposition}

\newtheorem{lem}[thm]{Lemma}
\newtheorem{cor}[thm]{Corollary} 

\theoremstyle{definition}
\newtheorem{defi}[thm]{Definition}
\newtheorem*{defi*}{Definition}

\theoremstyle{remark}

\newtheorem{rem}[thm]{Remark}
\newtheorem*{rem*}{Remark}
 
\newtheorem{ex}[thm]{Example}

\numberwithin{equation}{thm}

\newcommand{\R}{\mathbb{R}}
\newcommand{\Q}{\mathbb{Q}}
\newcommand{\Z}{\mathbb{Z}}
\newcommand{\N}{\mathbb{N}}

\newcommand{\Spec}{\mathrm{Spec}}

\newcommand{\HH}{\mathrm{H}}
\newcommand{\D}{\mathrm{D}}

\newcommand{\mc}{\mathcal}

\newcommand{\Dd}{\mathrm{D}}

\newcommand{\colim}{\mathrm{colim}}

\newcommand{\Hom}{\mathrm{Hom}}

\newcommand{\F}{\mathbb{F}}

\DeclareMathOperator{\sHom}{\mathscr{H}\text{\kern -3pt {\calligra\large om}}\,}

\newcommand{\Map}{\mathrm{Map}}

\newcommand{\et}{\mathrm{\acute{e}t}}\newcommand{\DM}{\mathrm{DM}^{\et}}

\newcommand{\Mod}{\mathrm{Mod}}

\newcommand{\Sh}{\mathrm{Sh}}

\newcommand{\catinfty}{\mathrm{Cat}_\infty}

\newcommand{\DMgm}{\mathrm{DM}^{\et}_{\mathrm{gm}}}
\newcommand{\DMc}{\mathrm{DM}^{\et}_{c}}

\newcommand{\rar}{\to}

\usepackage{appendix}

\makeatletter
\newcommand{\subjclass}[2][2020]{%
  \let\@oldtitle\@title%
  \gdef\@title{\@oldtitle\footnotetext{#1 \emph{Mathematics subject classification.} #2}}%
}
\newcommand{\keywords}[1]{%
  \let\@@oldtitle\@title%
  \gdef\@title{\@@oldtitle\footnotetext{\emph{Key words and phrases.} #1.}}%
}
\makeatother

\begin{document}
\maketitle

\begin{abstract}
  Over qcqs finite-dimensional schemes, we prove that étale motives of geometric origin can be characterised by a constructibility property which is purely categorical, giving a full answer to the question ``Do all constructible étale motives come from geometry?'' which dates back to Cisinski and Déglise's work. We also show that they afford the continuity property and satisfy h-descent and Milnor excision.
\end{abstract}
 \tableofcontents

\subsection{Introduction.}

Consider a scheme $X$ and a derived\footnote{We define our $\infty$-categories with derived coefficients so that we may use derived rings of the form $\Lambda\otimes_\Z^L\Z/n\Z$. However, all main theorems are stated for ordinary commutative rings as the extension to derived rings with heavy boundedness assumptions seems artificial.} commutative ring of coefficients $\Lambda$. We denote by $\DM(X,\Lambda)$ the stable $\infty$-category of étale motives defined as the Tate-stabilisation of the category of $\mathbb{A}^1$-local étale hypersheaves of $\Lambda$-modules over the category of smooth $S$-schemes. This version of étale motives without transfers was studied by Ayoub in \cite{MR3205601}; it is equivalent over Noetherian schemes to the category of h-motives studied by Cisinski and Déglise in \cite{MR3477640}. It is endowed with the six functors formalism by \cite[Theorem 4.6.1]{MR4466640}. Denote by $\Lambda_X$ the unit object of $\DM(X,\Lambda)$.

In this paper we explore some properties of étale motives of geometric origin:

\begin{defi*}
  Let $X$ be a scheme and let $\Lambda$ be a commutative ring. Let $M$ be an object of $\DM(X,\Lambda)$. We say that the motive $M$ is \emph{of geometric origin} (or simply \emph{geometric}) if it belongs to the thick subcategory $\DMgm(X,\Lambda)$ of $\DM(X,\Lambda)$ generated by the $f_\sharp(\Lambda_Y)(n)$ for $f\colon Y\to X$ smooth.
\end{defi*}

An ind-regular ring is a filtered colimit of regular rings. Our main goal is to prove that motives of geometric origin form an étale hypersheaf and to give a purely categorical characterisation of geometric étale motives:
\begin{mainthm}(\Cref{consgmDM} and \Cref{etlocDM})
  Let $\Lambda$ be a commutative ind-regular ring. 
  \begin{enumerate}    
    \item Let $X$ be a qcqs finite-dimensional scheme. An étale motive $M$ over $X$ is of geometric origin if and only if it is \emph{constructible}, that is for any open affine $U\subseteq X$, there is a finite stratification $U_i\subseteq U$  made of constructible locally closed subschemes such that each $M|_{U_i}$ is dualisable.
    \item The functor $\DMgm(-,\Lambda)$ is an étale hypersheaf over qcqs finite-dimensional schemes.
  \end{enumerate}
\end{mainthm}

This question dates back to \cite{MR3477640}, where Cisinski and Déglise considered étale-locally geometric\footnote{Beware that in \emph{loc. cit.}, they call \emph{constructible} what we chose to call \emph{of geometric origin} in this paper. We chose this terminology because they are the motives that can be built from objects coming from geometry. In addition, objects that are called \emph{locally constructible} by Cisinski and Déglise are more similar to classical constructible sheaves. In any case, the purpose of this paper is to show that all those notions coincide.} motives and showed that over Noetherian schemes, the latter coincides with constructible motives. This result does not hold for a general ring of coefficients (see \Cref{Contrex}).

To prove the above theorem, we need to generalise to non-Noetherian schemes the continuity properties proved in \textit{loc. cit.}

\begin{thm*}(Continuity; \Cref{continuitygmDM})
  Let $\Lambda$ be a commutative ring. Consider a qcqs finite-dimensional scheme $X$ that is the limit of a projective system of qcqs finite-dimensional schemes $(X_i)$ with affine transition maps.
    Then, the natural map
    $$\colim_i \hspace{1mm}\DMgm(X_i,\Lambda) \to \DMgm(X,\Lambda)$$
    is an equivalence. 
\end{thm*}

Using continuity, the most critical case of the main theorem is the case where $X=\Spec(k)$ with $k$ a field and where $\Lambda$ is regular. In that particular case, the result can be deduced from the case of rational coefficients which was proved in \cite{MR3477640} combined with results of \cite{ruimyAbelianCategoriesArtin2023} on Artin motives. The latter are consequences of results of \cite{balmer-gallauer} which imply that étale sheaves of the form $f_\sharp \Lambda_L$ for $L/k$ a finite separable extension, generate the category of étale-locally perfect constant sheaves as a thick subcategory of itself. 

We end the paper with some non-trivial descent properties of geometric étale motives. Namely we prove that they satisfy h-descent in the sense of \cite[Tag 0ETQ]{stacks-project} generalizing \cite[Corollary 5.5.7.]{MR3477640} to the non-Noetherian setting and we prove that they satisfy Milnor excision in the sense of \cite{zbMATH07335443}, giving a generalisation of the analogous of \cite[Theorem 5.6]{MR4319065} for geometric motives:

\begin{thm*}(\Cref{hdesc} and \Cref{milnor})
  Let $\Lambda$ be a commutative ring (\emph{resp.} an ind-regular commutative ring). The functor $\DMgm(-,\Lambda)$ satisfies Milnor excision (\emph{resp.} h-descent) over qcqs finite-dimensional schemes.
\end{thm*}

\begin{rem*}
  In an upcoming paper, we will define a category of geometric Nori motives with coefficients in an arbitrary commutative ring $\Lambda$ over qcqs finite-dimensional schemes and show that it satisfies arc-descent. We therefore expect $\DMgm(-,\Lambda)$ to satisfy arc-descent. 
\end{rem*}

\subsection{First descent properties.}
Until the end of the proof of the main theorem, we will use the following notation:
\begin{defi}
  \label{defDMc}
  Let $X$ be a scheme and let $\Lambda$ be a derived\footnote{As mentioned above, we make the definition here for a derived ring so that we may use it for torsion rings of the form $\Lambda\otimes^L_\Z\Z/n\Z$ with $\Lambda$ a commutative ring.} ring. We denote by $\DMc(X,\Lambda)$ the subcategory of $\DM(X,\Lambda)$ made of \emph{constructible motives}, that is those motives such that for any open affine $U\subseteq X$, there is a finite stratification $U_i\subseteq U$  made of constructible locally closed subschemes such that each $M|_{U_i}$ is dualisable.
\end{defi}

\begin{rem}
  If $X$ is Noetherian, then if $K$ is an object of $\DMc(X,\Lambda)$, we can find a finite stratification $(X_i)$ of $X$ such that each $K_{\mid X_i}$ is dualisable.
\end{rem}

\begin{rem}
  The functors of type $f^*$, $f_\sharp$ and $\otimes$ readily preserve geometric objects. The functor $i_*$, for $i$ a closed immersion, preserves geometric objects by using the localisation triangle.
  
  On the other hand, the functors of type  $f^*$ and $\otimes$ preserve constructible objects. The functors $j_\sharp$ and $i_*$, for $j$ an open immersion and $i$ a closed immersion, also preserve constructible objects.
\end{rem}

We first prove descent properties for geometric objects and for constructible objects. 

\begin{lem}\label{descentDM}
Let $\Lambda$ be a commutative ring. The functor $\DMgm(-,\Lambda)$ is a Nisnevich hypersheaf over qcqs schemes.
\end{lem}
\begin{proof}
  It suffices to show that $\DMgm(-,\Lambda)$ it is a subsheaf of $\DM(-,\Lambda)$ (a subsheaf of a hypersheaf is itself a hypersheaf for qcqs sites). Assume that $(U,V)$ is a Nisnevich cover of $X$ with $U$ open and $V\to X$ étale and let $M$ be a motive over $X$. Then if $M|_U$ and $M|_V$ are geometric, so is $M|_{U\times_X V}$, so that by Nisnevich excision, the motive $M$ itself is geometric.
\end{proof}

\begin{lem}\label{etlocDM}
  Let $\Lambda$ be a commutative ring. The functor $\DMc(-,\Lambda)$ is an étale hypersheaf over qcqs schemes.
\end{lem}
\begin{proof}
  It suffices to show that $\DMc(-,\Lambda)$ is a subsheaf of $\DM(-,\Lambda)$. The proof is very similar to the proof of \cite[Lemma 4.5]{MR4609461}, so we only give a sketch.
  
  Take a qcqs scheme $X$ and an étale-locally constructible motive $M$ over $X$. We can assume $X$ to be affine and to have a surjective étale map $f\colon Y\to X$ such that $f^*(M)$ is constructible. We then reduce to $f$ being finite étale using \cite[Tag 03S0]{stacks-project} and then to the case where $f$ is a torsor under some finite group $G$: writing $X$ as a limit of Noetherian schemes, the map $f$ is the pullback of some étale map defined over a Noetherian scheme which we can then replace by a torsor under some finite group. 
  
  Finally, we can construct a $G$-equivariant stratification of $Y$ from the stratification $\mc{Y}$ that witnesses the constructibility of $f^*(M)$. This yields a stratification of $X$ whose pullback to $Y$ is $\mc{Y}$. Hence, replacing $X$ with one of its strata, we can assume that $f^*(M)$ is dualisable; as being dualisable is étale-local, the motive $M$ is also dualisable which finishes the proof.
  %
\end{proof}

\subsection{Continuity.}

We now  prove the continuity property for geometric motives. We first need some compatibilities with rationalisation: 

\begin{prop}\label{chgbasratDM}
    Let $\Lambda$ be a commutative ring, let $f\colon Y\to X$ be a morphism of qcqs finite-dimensional schemes, let $M$ and $N$ in $\DM(X,\Lambda)$ be motives over $X$ with $M$ constructible or geometric and let $P$ be a motive over $Y$. Then, the natural transformations below are equivalences:
    \begin{enumerate}
      \item $\Map(M,N)\otimes \Q \to \Map(M,N\otimes \Q)$.
      \item $\underline{\Hom}(M,N)\otimes \Q \to \underline{\Hom}(M,N\otimes \Q).$
      \item $f_*(P)\otimes \Q \to f_*(P\otimes \Q)$.
      \item $i^!(N)\otimes \Q \to i^!(N\otimes \Q)$ (for $i$ a closed immersion).
    \end{enumerate}
  \end{prop}
  \begin{proof}
    We begin with the case where $M$ is geometric. In this proof, we say that $M$ is rationalisable if 
    $$\Map(M,N)\otimes \Q \to \Map(M,N\otimes \Q)$$ 
    is an equivalence. Rationalisable objects are closed under the functors of type $g_\sharp$. 
    Therefore, to prove (1), we can assume that $M=\Lambda_X$.
    Note first that if $K$ is a complex of étale sheaves of abelian groups, the canonical map:
    $$\HH^i_{\et}(X,K)\otimes \Q \to \HH^i_{\et}(X,K\otimes \Q)$$ is an isomorphism for any integer $i$. Indeed, \cite[Corollary 3.29]{MR4296353} gives a bound to the étale cohomological dimension with rational coefficients over qcqs finite-dimensional schemes, and the above assertion then follows from \cite[Lemma 1.1.10]{MR3477640}. By the same argument as \cite[Corollary 5.4.8]{MR3477640} this gives (1).

Geometric objects generate $\DM(X,\Lambda)$ as a localizing subcategory of itself. Hence to show that the maps of (2), (3) and (4) are equivalences, it suffices to show that they become equivalences after applying $\Map(C,-)$ for $C$ geometric. Since the functors of type $f^*$, $i_*$ and $\otimes$ preserve geometric objects, this follows by adjunction from the equivalence of (1).

  
  We now tackle the case where $M$ is constructible. The third equivalence implies that the functors of type $g^*$ preserve rationalisable objects. Hence, being rationalisable is Zariski-local: if we have a finite Zariski cover $\{U_i\to X\}$ and a motive $M$ over $X$ such that the $M|_{U_i}$ are rationalisable, so are the restrictions of $M$ over the intersections of the $U_i$; as we can write $M$ as a finite colimit with terms of the form $(V\to X)_\sharp(M|_{V})$ with $V$ an intersection of the $U_i$, the motive $M$ is also rationalisable. Hence, we can assume $X$ to be affine. 
  
  Note now that (4) implies that rationalisable objects are preserved by the functor $i_*$ for $i$ a closed immersion. By dévissage and localisation, we are therefore reduced to the case where $M$ is dualisable. In that case, replacing $N$ by $N\otimes M^\vee$ we can further reduce to the case where $M=\Lambda_X$ which we already proved. 
  
  Assertion (2) in the constructible setting can be proved in the same way as in the geometric setting. 
  \end{proof}

  For $\Lambda$ a derived ring, denote by $\Sh(X_\et,\Lambda)$ the $\infty$-category of étale hypersheaves of $\Lambda$-modules.
  The above proposition gives a control on the rationalised part of our categories, and the following theorem will be the control on the torsion part. 

  \begin{thm}[Bachmann]
    \label{rigidity}
    Let $n\in\N^*$ be an integer, let $\Lambda$ be a derived $\Z/n\Z$-algebra and let $X$ be a qcqs scheme such that $n$ is invertible on $X$. The natural functor 
    $$\rho_!\colon \Sh(X_\et,\Lambda)\to\DM(X,\Lambda)$$ is an equivalence. We will denote by $\Dd_\mathrm{cons}(X_\et,\Lambda)$ the inverse image of $\DMc(X,\Lambda)$ under this equivalence: as $\rho_!$ is symmetric monoidal it has the same description as \Cref{defDMc}, we call those \emph{constructible étale sheaves}.
  \end{thm}
  \begin{proof}
    Thanks to {\cite[Theorem 3.1]{bachmannrigidity}} the above functor is an equivalence if $\Lambda = \Z/n\Z$, and then this form of the theorem follows by taking $\Lambda$-modules on both sides.
  \end{proof}

  \begin{lem}[Martini-Wolf {\cite[Corollary 7.4.12]{martini_presentable_2022}}]
    \label{duallocconst}
    Let $\Lambda$ be a derived ring and let $X$ be a qcqs scheme. An object $K$ of $\Sh(X_\et,\Lambda)$ is dualisable if and only if there exists an étale covering $f\colon U\to X$ such that $f^*K$ is the constant sheaf associated to a perfect complex. 
  \end{lem}

  \begin{prop}
    \label{contDcons}
    Let $\Lambda$ be a commutative ring, let $R$ be a perfect derived $\Lambda$-algebra and let $X$ be a qcqs scheme. Assume that $X=\lim_{i\in I} X_i$ is a limit of qcqs schemes with affine transitions. Then the natural functor 
    $$\colim_i\Dd_\mathrm{cons}((X_i)_\et,R)\to \Dd_\mathrm{cons}(X_\et,R)$$ is an equivalence, where the colimit is taken in $\catinfty$.
  \end{prop}
  \begin{proof}
    This is an adaptation of \cite[Proposition 5.10]{MR4278670}.
    Up to writing each $X_i$ as a limit of Noetherian schemes with affine transition, we can assume the $X_i$ to be Noetherian.
    We first show that the functor is fully faithful. Let $i_0\in I$. It is enough to prove that for any two complexes $K$ and $L$ in $\Dd_{\mathrm{cons}}((X_{i_0})_\et,R)$ the map 
    $$\colim_{i\geqslant i_0}\Map_{\Dd_\mathrm{cons}((X_i)_\et,R)}(K_{\mid X_i},L_{\mid X_i})\to\Map_{\Dd_{\mathrm{cons}}(X,\Lambda)}(K_{\mid X},L_\mathrm{\mid X})$$ is an equivalence. By dévissage and localisation, we may assume that $K$ is dualisable and then reduce to the case where $K=R$. By adjunction, we thus have to show that the map 
    $$\colim_{i\geqslant i_0}\mathrm{Map}_{\mathrm{Mod}_{\Lambda}}(\Lambda,\mathrm{R}\Gamma(X_i,L))\to\Map_{\mathrm{Mod}_{\Lambda}}(\Lambda,\mathrm{R}\Gamma(X,L))$$ is an equivalence. Now, as perfect $R$-modules are perfect $\Lambda$-modules because $R$ is perfect over $\Lambda$, the object $L$, seen as an object of $\Sh((X_0)_\et,\Lambda)$ is constructible, in particular bounded so that up to a dévissage, it is concentrated in one degree. The result then follows from the commutation of étale cohomology with colimits over the $(X_i)$. 

    For the essential surjectivity, by dévissage it is enough to prove that any dualisable object lies in the image. 
    By \Cref{duallocconst}, if $K$ is a dualisable object of $\Sh(X_\et,R)$ there exists an étale covering $f\colon U\to X$ and a perfect complex of $R$-modules $P$ such that $f^*(K)$ is the constant sheaf with value $P$. 
    By Zariski descent we may even assume that $X$ is affine and up to cover the source of $f$ by an affine Zariski cover, we may also assume that the source of $f$ is affine. Then by spreading out (\cite[Théorème 8.10.5, Proposition 17.7.8]{ega4}) the morphism $f$ is the pullback of some étale $f^i$ with target $X_i$. Let $(f_n)$ and $(f^i_n)$ be the \v{C}ech nerves of $f$ and $f^i$. As the image of $(f_n^i)_\sharp P$ in $\Dd_\mathrm{cons}(X,R)$ is $(f_n)_\sharp P$, fully faithfulness ensures that the simplicial object $((f_n)_\sharp P)_{n\in\Delta}$ is the image of a simplicial object $((f_n^i)_\sharp P)_{n\in\Delta}$ in $\Dd_\mathrm{cons}(X_i,R)$. By étale descent, we have that $K = \colim_n (f_n)_\sharp P$ in $\Dd_\mathrm{cons}(X,R)$. Moreover if we denote by $L = \colim_n (f_n^i)_\sharp P$ in $\Sh((X_i)_\et,R)$, we have that $(f^i)^*L = P$ hence $L$ is dualisable. As the pullback functor $p_i^*\colon \Sh((X_i)_\et,R)\to\Sh(X_\et,R)$ commutes with colimits, we obtain that $K$ is the image of $L$, which is a constructible object over $X_i$, finishing the proof.

  \end{proof}

  \begin{thm}(Continuity)\label{continuitygmDM} Let $\Lambda$ be a commutative ring. Consider a qcqs finite-dimensional scheme $X$ that is the limit of a projective system of qcqs finite-dimensional schemes $(X_i)$ with affine transition maps.
    Then, the natural map
    $$\colim \hspace{1mm}\DMgm(X_i,\Lambda) \to \DMgm(X,\Lambda)$$
    is an equivalence. 
  \end{thm}

  \begin{proof}
    We can again assume that each $X_i$ is Noetherian.
    Fully faithfulness can be checked after tensorisation of the mapping spectra by $\Q$ and by $\Z/p\Z$. By \Cref{chgbasratDM}, the $-\otimes\Q$ passes inside the mapping spectra, and then our objects become compact by \cite[Lemma 5.1 (i)]{MR4319065} and this follows from continuity of the presentable category as in \cite[Lemma 2.5.11]{MR4466640}. Modding out by $p$ passes in the mapping spectra; hence, noting that $\Lambda\otimes_\Z \Z/p\Z$ is a perfect $\Lambda$-algebra, the result will then follow from the rigidity theorem in the form of \Cref{rigidity} and \Cref{contDcons}, once we know that geometric objects are constructible with torsion coefficients. This last assertion is true over $\Z/p\Z$ by \cite[Theorem 6.3.26]{MR3477640}; as the functor $-\otimes_\Z\Lambda$ preserves constructible étale sheaves and as the generators of $\DMgm(X,\Lambda\otimes_\Z\Z/p\Z)$ are of the form $M\otimes_\Z\Lambda$ for some $M$ in $\DMgm(X,\Z/p\Z)$, the result is true with $\Lambda\otimes_\Z \Z/p\Z$-coefficients. 
    
    To prove the essential surjectivity, first using that the \v{C}ech nerve of a covering coming from a finite open covering is finite, we can reduce to $X$ affine using Zariski descent. Then Zariski descent again (but this time on the functor that sends $g\colon Y\to X$ smooth to $g_\sharp g^*\Lambda_X$) implies that it suffices to reach $f_\sharp(\Lambda_Y)(n)$ with $Y$ and $X$ affine and $f\colon Y\to X$ smooth. The usual spreading out properties of \cite[Théorème 8.10.5, Proposition 17.7.8]{ega4} then give the result.
  \end{proof}

  \begin{cor}
    \label{PFproper}
    Let $\Lambda$ be a commutative ring and let $X$ be a finite-dimensional qcqs scheme.
    \begin{enumerate}
      \item Let $f\colon Y\to X$ be a separated finite type morphism. Then the functor $f_!$ restricts to a functor $$f_!\colon\DMgm(Y,\Lambda)\to\DMgm(X,\Lambda)$$ between geometric motives.
      \item The $\infty$-category $\DMgm(X,\Lambda)$ is the thick subcategory of $\DM(X,\Lambda)$ generated by the $p_*(\Lambda_Y)(n)$ for $p\colon Y\to X$ projective and $n\in\Z$.
      \item The $\infty$-category $\DM(X,\Lambda)$ is the localising subcategory of itself generated by the $p_*(\Lambda_Y)(n)$ for $p\colon Y\to X$ projective and $n\in\Z$.
    \end{enumerate}
  \end{cor}

  \begin{proof}
    For the first point, assume first that $f=p$ is proper.
    Let $M$ be an object of $\DMgm(Y,\Lambda)$, and write $X=\lim X_i$ as a limit of schemes that are of finite type over the integers. By \cite[Théorème 8.10.5.(xii)]{ega4} one can find a $p_{i_0} \colon Y_{i_0}\to X_{i_0}$ proper such that the pullback to $X$ is $p$. In particular $$Y = \lim_{i\geqslant i_0}Y_{i_0}\times_{X_{i_0}}X_i$$ and there is some $i$ such that $M$ is the pull back of a motive $N$ over $Y_i$ by continuity. Then by \cite[Proposition 4.2.11]{MR3971240} the object $(p_i)_*M_i$ is geometric in $\DM(X_i,\Lambda)$. Thus, its pullback to $X$ is geometric, but by the proper base change theorem this is $p_*M$. Finally, if $f$ is separated of finite type, choosing a Nagata compactification of $f$ yields the result.

    Now for the second point, write $X = \lim_i X_i$ as a limit of schemes that are of finite type over the integers. Let $M$ be an object of $\DMgm(X,\Lambda)$, that comes by continuity above from an object $M_i$ of $\DMgm(X_i,\Lambda)$ for some index $i$. Thanks to \cite[Lemme 2.2.23]{MR2423375} we know that the result is true for $\DMgm(X_i,\Lambda)$, so that there exists a finite diagram 
    $$F\colon J\to\DMgm(X_i,\Lambda)$$ such that $F(j)$ is of the form 
    $p^j_*\Lambda(n_j)[m_j]$ for $p^j$ projective and $m_j$ and $n_j$ integers, and such that $M_i$ is a direct factor of $\colim_J F$. Thus, the object $M$ is a direct factor of $(\colim_J F)_{\mid X}$, which by proper base change is the colimit of the objects $q^j_*\Lambda(n_j)[m_j]$, with $q^j$ the base change of $p^j$ to $X$. This finishes the proof, as the third point readily follows from the second one.
  \end{proof}

  As a consequence of continuity, we see that geometric motives are constructible.

  \begin{cor}\label{geoareconsDM}
    Let $\Lambda$ be a commutative ring. Geometric motives with coefficients $\Lambda$ are constructible. 
   \end{cor}
   \begin{proof}
    Let $X$ a qcqs scheme.
Write $X$ as a limit of Noetherian finite-dimensional schemes $X_i$ with affine transition maps. Any geometric motive with coefficients $\Z$, over $X_i$ is constructible using \cite[Theorem 6.3.26]{MR3477640} then this is also true with coefficients $\Lambda$ by tensoring by $\Lambda$ over $\Z$. If $M$ is a geometric object, it is the pullback of some $M_i$ over some $X_i$. As the motive $M_i$ is constructible, so is $M$. 
   \end{proof}

We now prove the converse when $X$ is $\Lambda$-finite\footnote{See the beginning of \cite[Section 5]{MR4319065} for the precise definition. This hypothesis implies that $\DM(X,\Lambda)$ is compactly generated and any finite dimensional scheme is $\Lambda$-finite if $\Lambda$ is a $\Q$-algebra.} by linking both notions to compactness: 
\begin{prop}\label{conscompDM} Let $\Lambda$ be a commutative ring. Let $X$ be a qcqs $\Lambda$-finite scheme and let $M$ be an object of $\DM(X,\Lambda)$. Then, the following are equivalent:
  \begin{enumerate}
    \item $M$ is of geometric origin.
    \item $M$ is constructible.
    \item The functor $\underline{\Hom}(M,-)$ is compatible with colimits.
    \item $M$ is a compact object of $\DM(X,\Lambda)$.
  \end{enumerate}
\end{prop}
\begin{proof}
  The equivalence between geometric objects and compact objects is \cite[Lemma 5.1 (i)]{MR4319065}.
On the other hand, the formula 
  $$\Map(f_\sharp(\Lambda_Y)(n)\otimes M,-)=\Map(f_\sharp(\Lambda_Y)(n),\underline{\Hom}(M,-))$$
  for $f\colon Y\to X$ smooth and $n$ an integer, combined with the fact that $f_\sharp(\Lambda_Y)(n)$ is compact implies that conditions (3) and (4) are equivalent.

We proved in \Cref{geoareconsDM} that motives of geometric origin are constructible. We now prove the converse.  Let $M$ be constructible over $X$. 
We can assume that $X$ is affine by Zariski hyperdescent of geometric motives. By dévissage and localisation, we can further assume that $M$ is dualisable. In that case $\underline{\Hom}(M,-)$ is compatible with colimits which finishes the proof.
\end{proof}

From this, we deduce continuity for constructible motives. 
We begin with a general lemma:
\begin{lem}\label{bestlemmaDM}
  Let $\Lambda$ be a commutative ring and let $M$ be an object of $\DMc(X,\Lambda)$ with $X$ a qcqs finite-dimensional scheme. Then, there is a geometric motive $P$ as well as a map $P\to M\oplus M[1]$ whose cone $C$ is such that $C/n\simeq C\oplus C[1]$ for some integer $n$. 
\end{lem}
\begin{proof}
  Take $M$ a constructible motive. \Cref{chgbasratDM} implies that $\DMgm(X,\Lambda\otimes\Q)$ is the idempotent-completion in $\DM(X,\Lambda \otimes \Q)$ of the rationalisation of $\DMgm(X,\Lambda)$. The characterisation of the idempotent completion \cite[Recollection 2.13]{balmer-gallauer} as well as \Cref{chgbasratDM} yields that $(M\oplus M[1])\otimes \Q$ belongs to the rationalisation of $\DMgm(X,\Lambda)$. This gives a geometric motive $P$ and a map $P\rar M\oplus M[1]$ which becomes an isomorphism after tensoring with $\Q$. Let $C$ be the cone of this map, then $C\otimes \Q=0$. Using \Cref{chgbasratDM} once again, we see that the multiplication by $n$ seen as a map $C\to C$ vanishes for $n$ large enough. This together with the exact triangle $C\xrightarrow{\times n} C \to C/n$ implies that $C/n\simeq C\oplus C[1]$.  
\end{proof}
\begin{cor}\label{continuityconsDM}
    Let $\Lambda$ be a commutative ring. Consider a qcqs finite-dimensional scheme $X$ that is the limit of a projective system of qcqs finite-dimensional schemes $(X_i)$ with affine transition maps.
     Then, the natural map
     $$\colim_i \hspace{1mm}\DMc(X_i,\Lambda) \to \DMc(X,\Lambda)$$
     is an equivalence. 
\end{cor}
\begin{proof} We first prove that the functor above is fully faithful. This can be checked by tensoring the mapping spectra with $\Q$ and $\Z/p\Z$, and then it follows from \Cref{chgbasratDM}, \Cref{conscompDM} and \Cref{contDcons}.   

     We now prove the essential surjectivity of the functor.  Let $M$ be a constructible motive. We want to prove that $M$ is in the essential image of the functor $$\colim \hspace{1mm}\DMc(X_i,\Lambda) \to \DMc(X,\Lambda).$$ Write the exact triangle $P \to M \oplus M[1] \to C$ of the above lemma. \Cref{continuitygmDM} implies that $P$ belongs to the essential image. As constructible motives are idempotent-complete, we can assume that $M$ is of the form $C/n$. But then, this follows from continuity with torsion coefficients which is a combination of \Cref{rigidity} and \Cref{contDcons}.
\end{proof}
\subsection{Constructible motives are of geometric origin.}

\begin{thm}\label{consgmDM} Let $\Lambda$ be a commutative ind-regular ring and let $X$ be a qcqs  finite-dimensional scheme. Then $$\DMgm(X,\Lambda)=\DMc(X,\Lambda).$$
\end{thm}

\begin{lem} Let $\Lambda$ be a commutative ring. Let $k$ be a field of characteristic exponent $p$.
  Then, the Artin motive functor $$\rho_!\colon\Dd(k_{\et},\Lambda[1/p])\rar \DM(k,\Lambda[1/p])\xleftarrow{\sim}\DM(k,\Lambda)$$ is a fully faithful monoidal functor. Furthermore any torsion motive lies in its essential image. 
\end{lem}
\begin{proof}
  If the result is known for $\Lambda=\Z$, then the result for an arbitrary ring of coefficients $\Lambda$ follows from tensoring both sides with $\Mod_\Lambda$: this operation preserves full faithfulness according to  \cite[Lemma 2.14]{haineNonabelianBasechangeBasechange2022}. 
  The result then follows from \cite[Corollary 1.2.3.4, Corollary 2.2.2.7]{ruimyAbelianCategoriesArtin2023}.
\end{proof}
\begin{lem}
  \label{colimringDMc}
  Let $\Lambda=\colim\Lambda_i$ be a filtered colimit of rings and let $X$ be a Noetherian scheme. Then the natural functor 
  $$\colim_i\DMc(X,\Lambda_i)\to\DMc(X,\Lambda)$$ is an equivalence.      
\end{lem}
\begin{proof}
  This proof was communicated to the authors by Denis-Charles Cisinski.
  Fully faithfulness can be checked after tensorisation by $\Q$ and $\Z/p\Z$. The rationalised statement follows from \Cref{conscompDM} and the fact that $\Lambda\mapsto \mathrm{Mod}_\Lambda(\DM(X,\Z))$ commutes with filtered colimits (\cite[Theorem 4.8.5.11]{lurieHigherAlgebra2022}). The mod $p$ statement can be proven as follows: a dévissage reduces to showing the statement for dualisable objects, and then for maps out of the constant sheaf, and even from maps out of the constant sheaf $\Z$, and it is now a consequence of commutativity of étale cohomology with uniformly bounded below complexes as in \cite[Lemma 5.3]{MR4609461}. 
  
  Now for the essential surjectivity first using \Cref{bestlemmaDM}, as it is clear that geometric objects are in the image, we reduce to reaching an object of the form $C/n$ which is of the form $\rho_! K$ with $K$ in $\Dd_\mathrm{cons}(X,\Lambda\otimes_\Z\Z/n\Z)$. Then, a dévissage reduces to a dualisable object for which by \Cref{duallocconst} there exists an étale covering $f\colon U\to X$ such that $f^*(K)$ is constant with value a complex $P$ in $\mathrm{Perf}_{\Lambda\otimes_\Z\Z/n\Z}$. As in the proof of \Cref{contDcons} we see that $K = \colim_{n\in\Delta}(f_n)_\sharp P$ with $(f_n)$ the \v{C}ech nerve of $f$. As $P$ is of the form $Q\otimes_{\Lambda_i\otimes_\Z\Z/n\Z}(\Lambda\otimes_\Z\Z/n\Z)$ (this follows from the fact that the objects of $\mathrm{Perf}_{(-)}$ are the compact objects of $\mathrm{Mod}_{(-)}$ which commutes with colimts), we see that $K = (\colim_{n\in\Delta}(f_n)_\sharp Q)\otimes_{\Lambda_i\otimes_\Z\Z/n\Z}\Lambda\otimes_\Z\Z/n\Z$ is the image of a dualisable object.
\end{proof}

  \begin{lem}\label{consgmfieldDM}
    Let $\Lambda$ be a commutative ind-regular ring. The theorem holds when $X$ is the spectrum of a field $k$.
  \end{lem}
    \begin{proof}
      Let $M$ be a constructible motive. We have to show that $M$ is geometric. \Cref{bestlemmaDM} yields a geometric motive $P$ and a map $P\to M\oplus M[1]$ whose cone $C$ is such that $C\otimes_\Z \Q=0$. 

    Furthermore, as $C$ is dualisable, we have $$\underline{\Hom}(C,\Lambda_k)\otimes_\Z \Q=\underline{\Hom}(C,\Lambda_k\otimes_\Z \Q)=\underline{\Hom}(C\otimes \Q,\Lambda_k\otimes_\Z \Q)=0.$$ Hence, the dual of $C$ is also a torsion motive. It suffices to show that any such $C$ is geometric.

    Assume first that $\Lambda$ is a regular ring. Writing $C=\rho_! D$, then $D$ is a dualisable object of $\Dd(k_{\et},\Lambda[1/p])$. In particular, it belongs to the thick stable subcategory generated by the $f_\sharp\underline{\Lambda[1/p]}$ for $f$ finite étale according to \cite[Proposition 2.1.2.15]{ruimyAbelianCategoriesArtin2023}. Thus, its image $C$ through $\rho_!$ is geometric as $\rho_!$ is compatible with the functors $f_\sharp$ and $f^*$ for $f$ finite étale using \cite[§4.4.2]{MR3477640}.

    Now we let $\Lambda$ be an ind-regular commutative ring, we write it as $\Lambda=\colim_i\Lambda_i$ with each $\Lambda_i$ regular. By \Cref{colimringDMc} any object of $\DMc(k,\Lambda)$ comes from a object in $\DMc(k,\Lambda_i)$ for some $i$, hence is geometric by the above paragraph.
    \end{proof}

    \begin{proof}(of \Cref{consgmDM})
      Using \Cref{continuitygmDM} and \Cref{continuityconsDM}, we can assume $X$ to be Noetherian.
      This is now an adaptation of the arguments given in \cite[Lemma 6.3.25]{MR3477640}.
      Let $M$ be a constructible motive on $X$. By Noetherian induction, using the localisation property and the stability of geometric motives under pushforwards by closed immersions it suffices to show that there is an open immersion $j\colon U \to X$ such that $j^*(M)$ is geometric, \textit{i.e.} we can replace $X$ with any non-empty open subscheme. Therefore, we can assume $X$ to be irreducible. Let $\eta$ be the generic point of $X$. Consider the commutative diagram:
      $$\begin{tikzcd}
        \underset{U\text{ open, }\eta\in U}{\colim} \DMgm(U,\Lambda) \ar[d,"\alpha"]\ar[r,"u"] & \underset{U\text{ open, }\eta\in U}{\colim} \DMc(U,\Lambda)\ar[d,"v"]\\
        \DMgm(\eta,\Lambda) \ar[r,"\beta"]& \DMc(\eta,\Lambda).
      \end{tikzcd}$$
      Then, \Cref{consgmfieldDM} above implies that $\beta$ is an equivalence and \Cref{continuitygmDM} implies that $\alpha$ is an equivalence. On the other hand $u$ is fully faithful and $v$ is an equivalence by \Cref{continuityconsDM}. Hence, they are both equivalences. The fact that $u$ is an equivalence yields the result.
    \end{proof}

    \begin{ex}\label{Contrex}
   This result fails without the assumption that $\Lambda$ is regular: take $\Lambda=\F_2[x]/(x^2-1)$. Then, by rigidity, we have $$\DM(\Spec(\R),\Lambda) \underset{\rho_!}{\xleftarrow{\sim}} \Sh(\R_{\et},\Lambda)\simeq\D(\Lambda[C_2])$$ where $C_2$ is the cyclic group with $2$ elements. Taking dualisable objects, we get an equivalence between $\DMc(\Spec(\R),\Lambda)$ and the subcategory $\D_{\Lambda\text{-perf}}(\Lambda[C_2])$ of $\D(\Lambda[C_2])$ made of those complexes with a perfect underlying complex of $\Lambda$-modules. 

   If $f\colon X\to \R$ is smooth, this equivalence sends $M_\R(X)(n)$ to the complex $f_\sharp(\Lambda)(n)$. The latter is isomorphic to $(f_\sharp(\F_2)(n))\otimes_{\F_2}\Lambda$. As by \cite[Corollary 4.5]{balmer-gallauer}, $f_\sharp(\F_2)(n)$ belongs to the thick subcategory generated by $\F_2$ and $\F_2[C_2]$, the complex $f_\sharp(\Lambda)(n)$ therefore belongs to the thick subcategory $\D_\mathrm{perm}(\Lambda[C_2])$ generated by $\Lambda$ and $\Lambda[C_2]$. Hence, through this equivalence $\DMgm(\Spec(\R),\Lambda)$ corresponds to $\D_\mathrm{perm}(\Lambda[C_2])$.
   
   Consider the complex $\Lambda_x$ of $\Lambda[C_2]$-modules whose underlying complex is $\Lambda$ placed in degree $0$ and with the action of the generator of $C_2$ given by multiplication by $x$. The complex $\Lambda_x$ belongs to $\D_{\Lambda\text{-perf}}(\Lambda[C_2])$ but by \cite[Example 3.15]{MR4468990} it does not belong to $\D_\mathrm{perm}(\Lambda[C_2])$. Hence, the inclusion $\DMgm(\Spec(\R),\Lambda) \to \DMc(\Spec(\R),\Lambda)$ is strict. 

  In fact, the result fails even for a characteristic $0$ non-regular ring: take $\Lambda' = \Z[x]/(x^2-1)$, and consider the $\Lambda'[C_2]$-module $\Lambda'_x$, defined in the same way. It can be seen as an étale motive through the Artin motive functor $\rho_!\colon \Sh(\R_\et,\Lambda')\to\DM(\Spec(\R),\Lambda')$. This object is dualisable, and if it were geometric then its image under the functor $-\otimes_{\Lambda'}\Lambda$, which is $\Lambda_x$ would also be geometric, which is not.
    \end{ex}  

\subsection{Milnor excision and h-descent.}

\begin{prop}\label{cdh}
  Let $\Lambda$ be a commutative ring. The functor $\DMgm(-,\Lambda)$ satisfies cdh-descent over qcqs finite-dimensional schemes. Over qcqs schemes of finite valuative dimension, it also satisfies cdh hyperdescent.
\end{prop}
\begin{proof}
   Since the presentable $\infty$-category $\DM(-,\Lambda)$ satisfies Nisnevich descent, the localisation property and proper base change, it satisfies cdh-descent as the proof of \cite[Theorem 3.3.10]{MR3971240} applies in this level of generality. Thus $\DMgm(-,\Lambda)$ satisfies non effective cdh descent. The effectivity of descent follows from cdh-excision: a descent data of geometric motives gives us a motive as a finite limit of proper pushforwards, which is itself geometric by \Cref{PFproper} (1). For hyperdescent, this follows from the fact that the cdh topos is hypercomplete over qcqs schemes of finite valuative dimension by \cite[Corollary 2.4.16]{zbMATH07335443}.
\end{proof}

We now prove that geometric étale motives satisfy h-descent in the sense of Rydh. 

\begin{thm}\label{hdesc}
  Let $\Lambda$ be a commutative ind-regular ring. The functor $\DMgm(-,\Lambda)$ satisfies h-descent in the sense of \cite[Tag 0ETQ]{stacks-project} over qcqs finite-dimensional schemes.
\end{thm}
\begin{proof}
  By \cite[Theorem 2.9]{zbMATH06786946}, h-descent is equivalent to cdh-descent combined with fppf-descent. Also, fppf-descent is the same as finite flat descent together with Nisnevich descent, see \cite[Definition 2.5(15)]{zbMATH06444763}. Therefore it suffices to prove that $\DMgm(-,\Lambda)$ has descent for the finite flat topology. Let $f\colon T\to S$ be a finite morphism.

  By writing $S=\lim_i S_i$ where the $S_i$ are affine $\Z$-schemes of finite type, we can find an index $i_0$ such that $f$ is the pullback of a finite map $f_{i_0}\colon T_{i_0}\to S_{i_0}$ using \cite[Théorème 8.10.5]{ega4}. By Noetherian induction, we can find a finite stratification $\{S_{i_0}^1,\ldots, S_{i_0}^p \}$ of $S_{i_0}$ such that $S_{i_0}^{k}$ is open in $\bigcup_{k'\geqslant k}S_{i_0}^{k'}$ for each $k$ and $f_{i_0}$ is the composition of a finite étale map and a purely inseparable map over each $S_{i_0}^k$. Hence, pulling back this stratification to $S$, we get a stratification $\{S^1,\ldots, S^p\}$ of $S$ with the same properties.
  
First we prove non-effective descent. For this, it suffices to prove that for any objects $M$ and $N$ of $\DMgm(S,\Lambda)$, the morphism 
  $$\Map_{\DM}(M,N)\to \lim_{n\in\Delta} \Map_{\DM}(M,(f_n)_*f_n^*N)$$ is an equivalence, where for any non-negative integer $n$, we let $f_n\colon T_n\to S$ be the projection given by the \v{C}ech nerve of $f$. Assume first that we have a closed immersion $i\colon Y\to S$ with open complement $j\colon U\to S$ such that the above map is an equivalence for $(j^*M,j^*N)$ and $(i^*M,i^*N)$. Then using that limit of exact triangles are exact triangles we have a morphism of exact triangles of spectra:
  \[\begin{tikzcd}
    {\Map(M,j_\sharp j^*N)} & {\Map(M,N)} & {\Map(M,i_* i^*N)} \\
    {\lim\limits_{n\in\Delta} \Map(M,(f_n)_*f_n^*j_\sharp j^*N)} & {\lim\limits_{n\in\Delta} \Map(M,(f_n)_*f_n^*N)} & {\lim\limits_{n\in\Delta} \Map(M,(f_n)_*f_n^*i_*i^*N)}
    \arrow[from=1-1, to=1-2]
    \arrow[from=1-2, to=1-3]
    \arrow[from=1-2, to=2-2]
    \arrow[from=1-1, to=2-1]
    \arrow[from=2-1, to=2-2]
    \arrow[from=2-2, to=2-3]
    \arrow[from=1-3, to=2-3]
  \end{tikzcd}\]
 But we have canonical identifications $$(f_n)_*f_n^*j_\sharp j^*N\simeq j_\sharp (f_n)_*f_n^*j^*N\text{ and }(f_n)_*f_n^*i_*i^*N\simeq i_* (f_n)_*f_n^*i^*N,$$ thus the left and right vertical maps of the above diagram are equivalences by assumption, so that the middle one also. Therefore, by induction on the number $p$ of strata of $S$, this proves the full faithfulness by induction, given that we have descent along finite étale maps by \Cref{etlocDM} and that localisation ensures that $\DMgm(-,\Lambda)$ is an equivalence along purely inseparable maps.

 For the effectivity of descent, we also prove the result by induction. Keeping the same notations as before, we have the following morphism of cofiber sequences of $\infty$-categories:
 \[\begin{tikzcd}
	{\DMgm(Y,\Lambda)} & {\DMgm(X,\Lambda)} & {\DMgm(U,\Lambda)} \\
	{\lim\limits_{n\in\Delta}\DMgm(Y_n,\Lambda)} & {\lim\limits_{n\in\Delta}\DMgm(X_n,\Lambda)} & {\lim\limits_{n\in\Delta}\DMgm(U_n,\Lambda)}
	\arrow[from=1-1, to=2-1,"\sim"]
	\arrow[from=2-1, to=2-2,"\tilde{i_*}"]
	\arrow[from=2-2, to=2-3,"\tilde{j^*}"]
	\arrow["{j^*}", from=1-2, to=1-3]
	\arrow["{i_*}", from=1-1, to=1-2]
	\arrow[from=1-2, to=2-2]
	\arrow[from=1-3, to=2-3,"\sim"]
\end{tikzcd}\]
in which the left vertical map is an equivalence by induction, and the right vertical map is an equivalence by étale descent. The functors $\tilde{i_*}$ and $\tilde{j_\sharp}$ have left adjoint $\tilde{i^*}$ and $\tilde{j_\sharp}$ that are computed term wise. In particular as we have a cofiber sequence of stable $\infty$-categories, in $\lim_{\Delta}\DMgm(X_n,\Lambda)$ we have an exact sequence 
$$\tilde{j_\sharp}\tilde{j^*}\to\mathrm{Id}\to\tilde{i_*}\tilde{i^*},$$ which is compatible with the same exact sequence in $\DMgm(X,\Lambda)$. Thus, if $(M_n)$ is some descent data in $\lim_\Delta\DMgm(X_n,\Lambda)$, there are objects $N_U$ and $N_Y$ in $\DMgm(U,\Lambda)$ and $\DMgm(Y,\Lambda)$ such that they correspond to the restrictions of $(M_n)$ to $(U_n)$ and $(Y_n)$. Thus in $\lim_\Delta\DMgm(X_n,\Lambda)$ there is a map $\tilde{i_*}N_Y\to \tilde{j_\sharp}N_U[1]$ of fiber $(M_n)$, which by full faithfulness gives a similar map in $\DMgm(X,\Lambda)$, whose fiber has image $(M_n)$.
\end{proof}

We finish the paper by proving Milnor excision.
\begin{prop}\label{milnor}
    Let $\Lambda$ be a commutative ring. The functor $\DMgm(-,\Lambda)$ satisfies Milnor excision (\cite[Definition 3.2.2]{zbMATH07335443}) over qcqs finite-dimensional schemes.
\end{prop}
\begin{proof}
    First, using \cite[Lemma 3.2.6]{zbMATH07335443} any Milnor square (\textit{i.e.} of the form below, bicartesian, with $f$ affine and $i$ a closed immersion) of qcqs finite-dimensional schemes 
    $$\begin{tikzcd}W \ar[d,"g"]\ar[r,"k"] & Y \ar[d,"f"] \\
    Z \ar[r,"i"] & X 
    \end{tikzcd}$$
    can be approximated by Milnor squares of finitely presented $\Z$-schemes (and therefore of Noetherian schemes) such that $f$ is of finite type. By continuity (\Cref{continuitygmDM}), and as filtered colimits are compatible with finite limits in $\catinfty$, we can therefore assume our schemes to be Noetherian and finite-dimensional and $f$ to be of finite type.

    Because pushforwards preserve geometric motives over quasi-excellent schemes by \cite[Theorem 6.2.13]{MR3477640}, the functor 
    \begin{equation}\label{milnorfunctor}\DMgm(X,\Lambda)\to \DMgm(Y,\Lambda)\times_{\DMgm(W,\Lambda)}\DMgm(Z,\Lambda) \end{equation} admits a right adjoint. By localization, the family of functors $(f^*,i^*)$ is conservative thus the above functor is conservative. To show that it is an equivalence, it suffices to show that its right adjoint is fully faithful.
    Thus take geometric motives $M_Y$, $M_Z$ and $M_W$ over $Y$, $Z$ and $W$ such that we have equivalences $k^*(M_Y)\simeq M_W \simeq g^*(M_Z)$. Let $M$ be the fiber of the map $f_*(M_Y)\oplus i_*(M_Z)\to f_*k_*(E_W)$. It is the image under the right adjoint of \cref{milnorfunctor} of the object $(M_Y,M_Z,W)$.  To finish the proof, it suffices to show that the natural maps $f^*(M)\to M_Y$ and $i^*(M)\to M_Z$ are equivalences. 
    Assume first that the ring $\Lambda$ is flat over the integers. Using \Cref{chgbasratDM}, we can therefore assume $\Lambda$ to be a $\Q$-algebra or to be a $\Z/p\Z$-algebra for some prime number $p$. The first case follows from \cite[Theorem 5.6, Lemma 2.1]{MR4319065} while the second case follows from \cite[Theorem 5.14]{MR4278670} and \cite[Corollary 3.2.11]{zbMATH07335443}.
    Now for a ring $\Lambda$ of finite type over the integers that we write as $R/I$ with $R$ a polynomial ring, we proceed as follows: as $\Lambda$ is perfect over $R$, the three motives $M_Y$, $M_Z$ and $M_W$ map to geometric objects of $\DMgm(-,R)$, together with descent data. In particular, the maps $f^*(M)\to M_Y$ and $i^*(M)\to M_Z$ are equivalences in $\DMgm(-,R)$ hence in $\DMgm(-,\Lambda)$. Thus the functor \cref{milnorfunctor} is an equivalence whenever $\Lambda$ is of finite type over the integers.
    Finally for a general ring, we can write it as a colimit of ring of finite type over the integers and with \Cref{colimringDMc} this is just a consequence of the commutation of filtered colimits with finite limits.
\end{proof}

\subsection*{Acknowledgments.}
The authors would like to thank Denis-Charles Cisinski very warmly for his interest in this work and his help on some technical points that have improved greatly the quality of the paper, in particular for mentioning the existence of \Cref{duallocconst} and for indications regarding the proof of \Cref{colimringDMc}. We thank Martin Gallauer for answering a question on \cite[Example 3.15]{MR4468990}.
We are grateful to Sophie Morel and Frédéric Déglise for their support and for useful comments. Finally, we would like to thank Fabrizio Andreatta, Federico Binda, Vova Sosnilo and Alberto Vezzani for some useful conversations.

We were funded by the ANR HQDIAG, RR was funded by the \emph{Prin 2022: The arithmetic of motives and L-functions} and ST had his PhD fundings given by the É.N.S de Lyon.

    \bibliographystyle{alpha}

\begin{thebibliography}{EHIK21b}

  \bibitem[AGV22]{MR4466640}
  Joseph Ayoub, Martin Gallauer, and Alberto Vezzani.
  \newblock The six-functor formalism for rigid analytic motives.
  \newblock {\em Forum Math. Sigma}, 10:Paper No. e61, 182, 2022.
  
  \bibitem[Ayo07]{MR2423375}
  Joseph Ayoub.
  \newblock Les six op\'{e}rations de {G}rothendieck et le formalisme des cycles
    \'{e}vanescents dans le monde motivique. {I}.
  \newblock {\em Ast\'{e}risque}, (314):x+466, 2007.
  
  \bibitem[Ayo14]{MR3205601}
  Joseph Ayoub.
  \newblock La r\'{e}alisation \'{e}tale et les op\'{e}rations de {G}rothendieck.
  \newblock {\em Ann. Sci. \'{E}c. Norm. Sup\'{e}r. (4)}, 47(1):1--145, 2014.
  
  \bibitem[BBD82]{MR0751966}
  A.~A. Beilinson, J.~Bernstein, and P.~Deligne.
  \newblock Faisceaux pervers.
  \newblock In {\em Analysis and topology on singular spaces, {I} ({L}uminy,
    1981)}, volume 100 of {\em Ast\'{e}risque}, pages 5--171. Soc. Math. France,
    Paris, 1982.
  
  \bibitem[BG22]{MR4468990}
  Paul Balmer and Martin Gallauer.
  \newblock Permutation modules and cohomological singularity.
  \newblock {\em Comment. Math. Helv.}, 97(3):413--430, 2022.

  \bibitem[BG23]{balmer-gallauer}
  Paul Balmer and Martin Gallauer.
  \newblock Finite permutation resolutions.
  \newblock {\em Duke Math. J.}, 172(2):201--229, 2023.
  
  \bibitem[BH21]{bachmannrigidity}
  Tom Bachmann and Marc Hoyois.
  \newblock Remarks on \'etale motivic stable homotopy theory, 2021.
  \newblock Available at 
  \newblock \url{https://arxiv.org/abs/2104.06002}.
  
  \bibitem[BM21]{MR4278670}
  Bhargav Bhatt and Akhil Mathew.
  \newblock The arc-topology.
  \newblock {\em Duke Math. J.}, 170(9):1899--1988, 2021.
  
  \bibitem[BS17]{zbMATH06786946}
  Bhargav Bhatt and Peter Scholze.
  \newblock Projectivity of the {Witt} vector affine {Grassmannian}.
  \newblock {\em Invent. Math.}, 209(2):329--423, 2017.
  
  \bibitem[CD16]{MR3477640}
  Denis-Charles Cisinski and Fr\'{e}d\'{e}ric D\'{e}glise.
  \newblock \'{E}tale motives.
  \newblock {\em Compos. Math.}, 152(3):556--666, 2016.
  
  \bibitem[CD19]{MR3971240}
  Denis-Charles Cisinski and Fr\'{e}d\'{e}ric D\'{e}glise.
  \newblock {\em Triangulated categories of mixed motives}.
  \newblock Springer Monographs in Mathematics. Springer, Cham, [2019] \copyright
    2019.
  
  \bibitem[CM21]{MR4296353}
  Dustin Clausen and Akhil Mathew.
  \newblock Hyperdescent and \'{e}tale {$K$}-theory.
  \newblock {\em Invent. Math.}, 225(3):981--1076, 2021.
  
  \bibitem[EGA IV]{ega4}
  A.~Grothendieck.
  \newblock {\'E}l{\'e}ments de g{\'e}om{\'e}trie alg{\'e}brique. {IV}: {\'E}tude
    locale des sch{\'e}mas et des morphismes de sch{\'e}mas. ({Troisi{\`e}me}
    partie). {R{\'e}dig{\'e}} avec la colloboration de {J}. {Dieudonn{\'e}}.
  \newblock {\em Publ. Math., Inst. Hautes {\'E}tud. Sci.}, 28:1--255, 1966.
  
  \bibitem[EHIK21a]{zbMATH07335443}
  Elden Elmanto, Marc Hoyois, Ryomei Iwasa, and Shane Kelly.
  \newblock Cdh descent, cdarc descent, and {Milnor} excision.
  \newblock {\em Math. Ann.}, 379(3-4):1011--1045, 2021.
  
  \bibitem[EHIK21b]{MR4319065}
  Elden Elmanto, Marc Hoyois, Ryomei Iwasa, and Shane Kelly.
  \newblock Milnor excision for motivic spectra.
  \newblock {\em J. Reine Angew. Math.}, 779:223--235, 2021.
  
  \bibitem[GK15]{zbMATH06444763}
  Ofer Gabber and Shane Kelly.
  \newblock Points in algebraic geometry.
  \newblock {\em J. Pure Appl. Algebra}, 219(10):4667--4680, 2015.

  \bibitem[HA]{lurieHigherAlgebra2022}
  Jacob Lurie.
  \newblock Higher {{Algebra}}, 2022.
  \newblock Available at
  \url{https://people.math.harvard.edu/~lurie/papers/HA.pdf}.

  \bibitem[Hai22]{haineNonabelianBasechangeBasechange2022}
  Peter~J. Haine.
  \newblock From nonabelian basechange to basechange with coefficients, 2022.
  \newblock arXiv:2108.03545 [math].

  \bibitem[HRS23]{MR4609461}
  Tamir Hemo, Timo Richarz, and Jakob Scholbach.
  \newblock Constructible sheaves on schemes.
  \newblock {\em Adv. Math.}, 429:Paper No. 109179, 46, 2023.
  
  \bibitem[HS23]{MR4630128}
  David Hansen and Peter Scholze.
  \newblock Relative perversity.
  \newblock {\em Comm. Amer. Math. Soc.}, 3:631--668, 2023.
  
  \bibitem[MW22]{martini_presentable_2022}
  Louis Martini and Sebastian Wolf.
  \newblock Presentable categories internal to an
    $\infty$-topos, 2022.
  \newblock Available at \url{https://arxiv.org/abs/2209.05103}.
  
  \bibitem[Rui22]{ruimyAbelianCategoriesArtin2023}
  Raphaël Ruimy.
  \newblock Integral {A}rtin motives: t-structures and {A}rtin {V}anishing
    {T}heorem, 2022.
  \newblock Available at \url{https://arxiv.org/abs/2211.02505}.
  
  \bibitem[Stacks]{stacks-project}
  {The Stacks Project Authors}.
  \newblock The {S}tacks project, 2018.
  \newblock Available at \url{https://stacks.math.columbia.edu/}.
  
  \end{thebibliography}

Raphaël Ruimy \url{raphael.ruimy@unimi.it}

Dipartimento di Matematica "F. Enriques" - Università degli Studi di Milano,

via Cesare Saldini 50, I-20133 Milano, Italy\\

Swann Tubach \url{swann.tubach@ens-lyon.fr}

E.N.S Lyon, UMPA, 

46 Allée d'Italie, 69364 Lyon Cedex 07, France
\end{document}